\documentclass[11pt]{amsart}
\usepackage{amsmath,amssymb,latexsym,soul, tikz}
\usepackage{cite,mathrsfs,amsfonts}
\usepackage{color,enumitem,graphicx}
\usepackage[colorlinks=true,urlcolor=blue,
citecolor=red,linkcolor=blue,linktocpage,pdfpagelabels,
bookmarksnumbered,bookmarksopen]{hyperref}
\usepackage[english]{babel}

\usepackage[left=2.7cm,right=2.7cm,top=2.8cm,bottom=2.8cm]{geometry}

\usepackage[hyperpageref]{backref}

\usepackage{mathrsfs}

\usepackage{lmodern}

\numberwithin{equation}{section}

\newtheorem{theorem}{Theorem}[section]
\newtheorem{lemma}[theorem]{Lemma}

\newtheorem{proposition}[theorem]{Proposition}
\newtheorem{corollary}[theorem]{Corollary}
\theoremstyle{definition}

\newcommand{\E}{{\mathbb E}}
\newcommand{\N}{{\mathbb N}}
\newcommand{\F}{{\mathbb F}}

\newcommand{\Z}{{\mathbb Z}}
\newcommand{\R}{{\mathbb R}}
\renewcommand{\S}{{\mathbb S}}
\newcommand{\eps}{\varepsilon}

\newcommand{\Hsnorm}[1]{\left\|#1\right\|_{\dot H^s}}
\newcommand{\Hsnormlarge}[1]{\biggl\|#1\biggr\|_{\dot H^s}}
\newcommand{\weakto}{\rightharpoonup}

\newcommand{\U}{U_{\varepsilon,z}}
\newcommand{\Uoz}{U_{1,0}}
\renewcommand{\u}{u_{\delta,\varepsilon,z}}
\newcommand{\uu}{u_{\bar{\varepsilon}^2,\bar{\varepsilon},z}}
\newcommand{\Ray}{\mathscr{R}}
\newcommand{\C}{\complement}

\newcommand{\dist}{\mathrm{dist}}

\newcommand{\p}{\varphi_{\delta,z}}
\newcommand{\wO}{\widetilde{\Omega}}

\newcommand{\B}{B_{3\delta}(L)}
\newcommand{\cal}{\mathcal}

\title[Nonlocal critical problems in contractible domains]{Nonlocal problems at critical \\ growth in contractible domains}

\author[S.\ Mosconi]{Sunra Mosconi}
\address[S.\ Mosconi]{Dipartimento di Informatica
\newline\indent
Universit\`a degli Studi di Verona
\newline\indent
C\'a Vignal 2, Strada Le Grazie 15, 37134 Verona, Italy}

\author[N.\ Shioji]{Naoki Shioji}
\address[N.\ Shioji]{Department of Mathematics
\newline\indent
Faculty of Engineering
\newline\indent
Yokohama National University
\newline\indent
Tokiwadai, Hodogaya-ku,
Yokohama 240-8501, Japan}
\email{shioji@ynu.ac.jp}

\author[M.\ Squassina]{Marco Squassina}
\address[M.\ Squassina]{Dipartimento di Informatica
\newline\indent
Universit\`a degli Studi di Verona
\newline\indent
C\'a Vignal 2, Strada Le Grazie 15, 37134 Verona, Italy}
\email{marco.squassina@univr.it}


\thanks{S. Mosconi and M. Squassina where partially supported by Gruppo Nazionale per l'Analisi Matematica, la Probabilit\`a e le loro Applicazioni (INdAM). N. Shioji is partially supported by the Grant-in-Aid for Scientific Research (C) (No. 26400160) from Japan Society for the Promotion of Science. Part of the paper was written during the XXV Italian Workshop on Calculus of Variations held in Levico, in February 2015.}

\begin{document}

\begin{abstract}
 We prove the existence of a positive solution for nonlocal problems
 involving the fractional Laplacian and a critical growth power nonlinearity
 when the equation is set in a suitable contractible domain.
\end{abstract}

\maketitle
\section{Introduction}
\subsection{Overview}
Let  $\Omega$ be a smooth bounded domain of $\R^N$ with $N\geq 3$. 
In the celebrated papers \cite{coron,bahricoron} A.\ Bahri and J.M.\ Coron showed the existence of solutions to the critical problem
\begin{equation}
\begin{cases}
\label{probcrit-s1}
-\Delta u=u^{\frac{N+2}{N-2}}, & \text{in $\Omega$,} \\
\,\, u>0 & \text{in $\Omega$},\\
\,\, u=0 & \text{on $\partial\Omega$,}
\end{cases}
\end{equation}
provided that
$H_m(\Omega,\Z_2)\not=\{0\}$ for some $m\in\N\setminus\{0\}$, where
$H_m(\Omega,\Z_2)$ denotes the homology of dimension $m$ of $\Omega$
with $\Z_2$-coefficients. Their result, in particular, always yields a solution to \eqref{probcrit-s1} in $\R^3$ provided that the 
domain $\Omega$ is {\em not} contractible, since
$H_1(\Omega,\Z_2)\not=\{0\}$ or $H_2(\Omega,\Z_2)\not=\{0\}$.
This is achieved via various sofisticated
arguments from algebraic topology. The results of \cite{coron,bahricoron} provide 
a sufficient but not necessary condition for the existence of solutions:
indeed, in \cite{dancer,ding,passaseo}, E.N.\ Dancer,  W.Y.\ Ding and D.\ Passaseo showed that problem \eqref{probcrit-s1}
admits nontrivial solutions also in suitable contractible domains.
Let  $N>2s$, $s\in (0,1)$ and consider the problem 
\begin{equation}
\begin{cases}
\label{probcrit}
(-\Delta)^s u=u^{\frac{N+2s}{N-2s}} & \text{in $\Omega$,} \\
\,\, u>0 & \text{in $\Omega$},\\
\,\, u=0 & \text{in $\R^N\setminus\Omega$},
\end{cases}
\end{equation}
involving the fractional Laplacian $(-\Delta)^s$.
Fractional Sobolev spaces are well known since the beginning of the last century, especially in the framework of harmonic analysis. On the other hand, recently, after the seminal paper of Caffarelli and Silvestre \cite{caffarelli},
a large amount of contributions appeared on problems
which involve the fractional diffusion $(-\Delta)^s$, $0<s<1$. Due to its nonlocal character, working on bounded
domains imposes to detect an appropriate variational formulation 
for the problem. We will consider functions 
on $\R^N$ with $u=0$ in $\R^N\setminus\Omega$ replacing the
usual boundary condition $u=0$ on $\partial\Omega$. More precisely, $\dot H^s(\R^N)$ denotes the space of functions $u\in L^{2N/(N-2s)}(\R^N)$ such that
$$
\int_{\mathbb{R}^{2N}}\frac{|u(x)-u(y)|^2}{|x-y|^{N+2s}}dxdy<\infty.
$$
It is known that $\dot H^s(\R^N)$ is continuously embedded into
$L^{2N/(N-2s)}(\R^N)$
and it is a Hilbert space, see e.g.\ \cite{valdserv-tams}.  For any $\Omega\subseteq \R^N$ we will set 
$$
X_\Omega=\Big\{u\in \dot H^s(\R^N): \text{$u=0$ in $\R^N\setminus\Omega$}\Big\},
$$ 
and say that $u\in X_\Omega$ weakly solves \eqref{probcrit} if
\begin{equation}
\int_{\mathbb{R}^N}  (-\Delta)^{s/2} u (-\Delta)^{s/2}\varphi\, dx
= \int_{\mathbb{R}^N} u^{\frac{N+2s}{N-2s}}\varphi\, dx,  
\,\,\,\quad \text{for all $\varphi\in X_\Omega$}.
\label{vfor-problema}
\end{equation}
It is natural to expect, as in the local case, that by assuming suitable
geometrical or topological conditions on $\Omega$ one can get solutions to \eqref{probcrit}.
To the best of our knowledge, the situation is the following:
\vskip2pt
\noindent
$\bullet$ if $\Omega$ is a star-shaped domain, 
then \eqref{probcrit} does {\em not} admit
solutions (see \cite{ros-oton});
\vskip2pt
\noindent
$\bullet$ if there is a point
$x_0\in \R^N$ and radii $R_2>R_1>0$ such that
\begin{equation}
\label{geom}
\{R_1\leq |x-x_0|\leq R_2\}\subset \Omega,\qquad
\{|x-x_0|\leq R_1\}\not\subset \overline\Omega,
\end{equation}
then \eqref{probcrit} admits
a solutions provided that $R_2/R_1$ is sufficiently large (see \cite{SSS}). \newline
Concerning nonexistence in star-shaped domains is still unknown if  
{\em sign-changing} solutions for the critical problem
can be ruled out as for the local case and this is connected with delicate unique
continuation results up to the boundary that are currently unavailable in this
framework. Concerning the existence of solutions under more general assumptions
than \eqref{geom}, like when $H_m(\Omega,\Z_2)\not=\{0\}$ for some $m\in\N\setminus\{0\}$, the result is expected but not available yet.

\subsection{Main result}
The goal of this paper is to provide a 
fractional counterpart of the results \cite{dancer,ding,passaseo} on the existence
of solutions in suitable contractible domains of $\R^N.$
More precisely, our main result is stated next, see Figure \ref{figball}. We will write $x=(x', x_N)\in \R^N$ for $x'\in \R^{N-1}$, $x_N\in \R$.
%
%
%

\begin{theorem}
\label{mainthm2}
Assume that $N\geq 3$ and $0<s<1$, or $N=2$ and $0<s\leq 1/2$ and let $0<R_0<R_1<R_2<R_3$.
Then problem \eqref{probcrit} admits a solution in any smooth domain $\Omega\subseteq B_{R_3}\setminus B_{R_0}$ satisfying
\begin{equation}
\label{thin-tunnel}
 \overline{\Omega}\cap \{(0, x_N): x_N\geq 0\}=\emptyset,
\end{equation}
\begin{equation}
\label{thin-tunnel2}
\Omega\supset \{R_1\leq |x|\leq R_2\}\setminus\{|x'|< \delta, x_N\geq 0\},
\end{equation}
for $\delta>0$ sufficiently small, depending only on $N, s$ and the $R_i$'s $i=0,\dots, 3$.
%
\end{theorem}

\begin{figure}
\centering
\begin{tikzpicture}[y=0.80pt, x=0.8pt,yscale=-1, inner sep=0pt, outer sep=0pt, scale=0.5]

\draw[color=gray, line width=5pt] (330, 770) circle[radius=120];
\fill[fill=white, draw opacity=1] (310, 640) rectangle (350, 660);
\draw[->] (330, 770) --  (375, 875);
\draw  (350, 860) node{$R_1$};
\draw[->] (330, 770) --  (400, 874);
\draw (400, 832) node{$R_2$};
\draw[->] (330,950) -- (330,600) node[above right]{$x_N$};
\draw[very thick] (330, 770) -- (330, 603);
\draw[->] (100, 770) -- (600, 770) node[below right]{$x'$};

  \path[draw=black, line join=miter,line cap=butt,even odd rule,line width=0.500pt]
    (372.0472,680.3150) .. controls (318.8976,644.8819) and (336.6142,627.1654) ..
    (372.0472,627.1654) .. controls (407.4803,627.1654) and (460.6299,644.8819) ..
    (478.3465,662.5984) .. controls (506.3588,690.6107) and (525.9503,734.6970) ..
    (525.9503,770.1300) .. controls (525.9503,805.5631) and (506.2319,891.0648) ..
    (453.0823,908.7813) .. controls (399.9327,926.4979) and (354.3307,928.3465) ..
    (318.8976,928.3465) .. controls (283.4646,928.3465) and (230.3150,910.6299) ..
    (212.5984,857.4803) .. controls (194.8819,804.3307) and (185.9454,773.0546) ..
    (185.9454,719.9050) .. controls (185.9454,666.7554) and (246.1829,629.1672) ..
    (263.8994,629.1672) .. controls (281.6160,629.1672) and (318.8976,627.1654) ..
    (318.8976,644.8819) .. controls (318.8976,662.5984) and (312.8889,668.1442) ..
    (297.6371,684.6284) .. controls (285.6052,697.6325) and (282.0757,717.9031) ..
    (282.0757,753.3362) .. controls (282.0757,788.7693) and (289.9330,841.1494) ..
    (325.3661,841.1494) .. controls (360.7992,841.1494) and (371.6799,821.6852) ..
    (378.3657,815.2691) .. controls (389.7638,804.3307) and (395.4163,789.4941) ..
    (409.0224,755.8010) .. controls (425.1968,715.7481) and (407.4803,698.0315) ..
    (372.0472,680.3150) -- cycle;

\draw (260, 800) node{$\Omega$};

\end{tikzpicture}

\caption{$\Omega$ contains a spherical shell minus a small cylindrical neighbourhood of its north pole, and must be distant from the positive $x_N$ axis.}
\label{figball}
\end{figure}
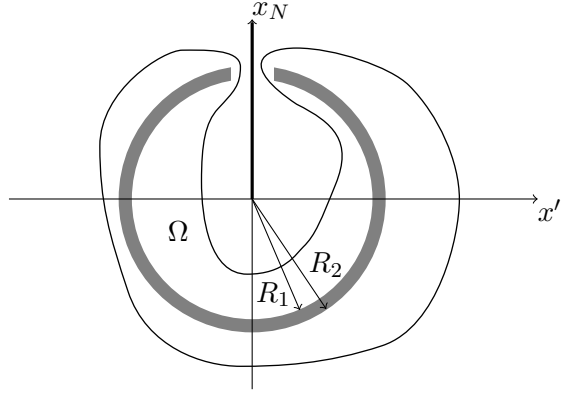

\noindent
We briefly describe the idea of the proof. We first consider solutions of problem \eqref{probcrit} as critical points of the free energy $I_\Omega:X_\Omega\to \R$
\[
I_\Omega(u)=\frac 1 2\int_{\R^N}|(-\Delta)^{\frac s 2} u|^2\, dx-\int_{\R^N} |u|^{2^*}\, dx, \qquad 2^*=\frac{2N}{N-2s},
\]
where we denote $I_{\R^N}=I$ for brevity, on the Nehari manifold 
\[
{\cal N}_+(\Omega)=\{u\in X_\Omega: u\geq 0,\ I'(u)(u)=0\}.
\]
We look at critical points near the {\em minimal energy} 
\[
\inf_{{\cal N}_+(\Omega)} I_\Omega=\inf_{{\cal N}_+(\R^N)} I=:c_\infty>0,
\]
and proceed by contradiction, assuming there is no critical point for $I_\Omega$ in $]c_\infty, 2c_\infty[$. Through a regularity lemma and known results we rule out the existence of nonnegative, nontrivial weak solutions to \eqref{probcrit} in the half-space. Then we can apply the characterization of Palais-Smale sequences proved in \cite{pal-pis2} to get that the $(PS)_c$ condition holds for all $c\in \ ]c_\infty, 2c_\infty[$. The contradiction will arise through a deformation argument near a minimax level, constructed as follows.
\vskip3pt
\noindent
In the whole $\R^N$, the solutions of problem \eqref{probcrit} are of the form
\begin{equation}
\label{Talentiane}
\U (x)=d_{N, s}\left(\frac{\eps}{\eps^2+|x-z|^2}\right)^{\frac{N-2s}{2}}
\end{equation}
for arbitrary $\eps>0$, $z\in \R^N$, and these solutions minimize $I$
on the Nehari manifold ${\cal N}_+(\R^N)$. Notice that, for $\eps\to 0$, most of the energy of $\U$ concentrates arbitrarily near $z$. Letting $R=(R_1+R_2)/2$, this enables us to cut-off $\U$ near each $z\in {\mathbb S}_R^{N-1}=\{|z|=R\}$ while keeping the energy almost minimal for sufficiently small $\eps\gg\delta>0$. Projecting onto ${\cal N}_+(\Omega)$ we thus obtain for any $z\in {\mathbb S}_R^{N-1}$ a function $v_z\in {\cal N}_+(\Omega)$ whose {\em barycenter} 
\[
 \beta (v_z)=\int_{\R^N} x\, v_z^{2^*}dx/\int_{\R^N} v_z^{2^*}dx
\]
well defines a map 
\[
{\mathbb S}^{N-1}_R\ni z\mapsto R\frac{\beta(v_z)}{|\beta(v_z)|}\in {\mathbb S}^{N-1}_R.
\]
Since $v_z$ is obtained cutting off $\U$ near $z$ for very small $\eps$, its barycenter is near $z$ and the resulting map is near the identity, thus has Brouwer degree $1$. 
Therefore the minimax problem
\[
c=\inf_{\varphi \in \Gamma}\sup_{z\in {\mathbb S}^{N-1}_R} I_\Omega(\varphi(z)),\qquad \Gamma=\Big\{\varphi\in C^0( \mathbb{S}_R^{N-1},{\cal N}_+(\Omega)): {\rm deg}\big(R\frac{\beta(\varphi(\cdot))}{|\beta(\varphi(\cdot))|}, {\mathbb S}^{N-1}_R\big)\neq 0\Big\}
\]
is well defined, and its minimax value is almost minimal, in particular less than $2c_\infty$. It can be proven that $c$ is also strictly greater than $c_\infty$, due to the fact that the whole half-line $\{(0, x_N): x_N\geq 0\}$ is a positive distance apart from $\overline\Omega$. Therefore $c\in \ ]c_\infty, 2c_\infty[$ (where the $PS$ condition holds), and through classical variational methods we find that $c$ is a critical value, reaching the contradiction.

\vskip3pt
\noindent
The most delicate part of the argument is the construction of the cut-offs of $\U$ with almost minimal energy, and this is where the condition $s\in ]0,1/2]$ when $N=2$ arises. While this seems a technical limitation at first, it really depends on the fact that the Bessel capacity $B_{s,2}$ of segments  vanishes only when $N-2s\geq 1$ ($1$ being the Hausdorff dimension of segments). For $N\geq 3$, $s\in \ ]0,1[$ the capacity of a segment $L$ vanishes, thus any function can be cutted-off near $L$ paying an arbitrary small amount of energy in the process. This is indeed what has to be done to $\U$ near the segment $L=\{(0, x_N): x_N\geq 0\}$ missing from $\Omega$, at least when $z\in {\mathbb S}^{N-1}_R$ is near $L$ (e.g.\ $z=(0,R)\notin\Omega$). In the case $N=2$, which arises only in the non-local case, the "cutting-off almost preserving the energy" procedure for such $z$'s fails for $s\in \ ]1/2, 1[$, having $L$ locally  nonzero capacity.

\noindent

\subsection{Plan of the paper}
In Section \ref{Prelim} we collect various preliminary results.
In Section~\ref{Estimates} we derive careful estimates on the energy
of suitable truncations of the Talenti functions \eqref{Talentiane}. Finally, in Section~\ref{Existence}
we implement the topological argument using the results of Section~\ref{Estimates}.

\section{Preliminaries}
\label{Prelim}

Let for any $u\in X_\Omega=\{u\in \dot H^s(\R^N): u\equiv 0 \text{ on $\C\Omega$}\}$
\begin{equation}
\label{Ftransform}
[u]_s^2=\frac{C(N, s)}{2}\int_{\R^{2N}}\frac{(u(x)-u(y))^2}{|x-y|^{N+2s}}\, dx\, dy=\int_{\R^N}|\xi|^{2s}|{\mathcal F}u(\xi)|^2\, d\xi,
\end{equation}
where ${\cal F}$ is the Fourier transform and  
\[
C(N, s)=\left(\int_{\R^N}\frac{1-\cos (\xi_1)}{|\xi|^{N+2s}}\, d\xi\right)^{-1}.
\]
Clearly $[\ ]_s$ is a Hilbert norm on $X_\Omega$ with associated scalar product
\[
(u, v)_s:=\frac{C(N, s)}{2}\int_{\R^{2N}}\frac{(u(x)-u(y))(v(x)-v(y))}{|x-y|^{N+2s}}\, dx\, dy,\qquad u, v\in \dot H^s(\R^N).
\]
The $s$-fractional Laplacian of $u\in\dot H^s(\R^N)$ is the gradient of the functional
\[
\dot H^s(\R^N)\ni u\mapsto \frac{1}{2}[u]^2_s
\]
and can be identified, for $u\in C^\infty(\R^N)\cap \dot{H}^s(\R^N)$ with the function
\[
(-\Delta)^s u(x)=C(N,s){\rm PV}\int_{\mathbb{R}^N} \frac{u(x)-u(y)}{|x-y|^{N+2s}}\,dy=C(N, s)\lim_{\eps\to 0}\int_{\C B_\eps(x)}\frac{u(x)-u(y)}{|x-y|^{N+2s}}\,dy.
\]
In this regard, notice that  due to \cite[Proposition 2.12]{IMS2} (see also \cite[Definition 2.4]{IMS2})
\[
\int_{\C B_\eps(\cdot)}\frac{u(\cdot)-u(y)}{|\cdot-y|^{N+2s}}\,dy\longrightarrow (-\Delta)^s u \qquad \text{strongly in $L^1_{\rm loc}(\R^N)$ as $\eps\to 0$}.
\]

We recall now the following

\begin{proposition}[Hardy-Littlewood-Sobolev inequality]
Let $0<\lambda<N$, and $p>1$, $q>1$ satisfy
\[
\frac{1}{q}+\frac{1}{p}=1+\frac{\lambda}{N}.
\]
Then for any $u\in L^q(\R^N)$, $v\in L^p(\R^N)$ it holds
\[
\int_{\R^{2N}}\frac{|u(x)v(y)|}{|x-y|^{N-\lambda}}\, dx\, dy\leq C\|u\|_q\|v\|_p
\]
for some constant $C=C(N, s, p)$.
\end{proposition}

Using the sharp form of the Hardy-Littlewood-Sobolev inequality, 
in \cite{costav} it is proved that the fractional Sobolev inequality, for $s<\frac{N}{2}$,
\[
S(N, s)\|u\|_{2^*}^2\leq [u]_s^2, \qquad \forall u\in \dot{H}^s(\R^N), \qquad 2^*=\frac{2N}{N-2s}
\]
holds with sharp constant
\[
S(N, s)=2^{2s}\pi^{s}\frac{\Gamma\big(\frac{N+2s}{2}\big)}{\Gamma\big(\frac{N-2s}{2}\big)}\left(\frac{\Gamma\big({\frac{N}{2}}\big)}{\Gamma(N)}\right)^{\frac{2s}{N}}
\]
and equality holds if and only if
\[
u(x)=c\left(\frac{\eps}{\eps^2+|x-z|^2}\right)^{\frac{N-2s}{2}},\quad \text{for any $c\in \R$, $\eps>0$ and $z\in \R^N$}.
\]

\subsection{Nehari manifold}
We consider the functional 
\[
X_\Omega\ni u\mapsto I_\Omega(u)=\frac{1}{2}[u]_s^2-\frac{1}{2^*}\int_{\R^N} |u|^{2^*}\, dx, \qquad I_{\R^N}=I.
\]
Its critical points are the only solutions of 
\begin{equation}
\label{problem}
\begin{cases}
(-\Delta)^s u=|u|^{\frac{4s}{N-2s}}u&\text{in $\Omega$},\\
u=0&\text{in $\C\Omega$}
\end{cases}
\end{equation}
and the nontrivial ones belong to the associated Nehari manifold
\[
{\cal N}(\Omega):=\{u\in X_\Omega\setminus\{0\}: [u]_s^2=\|u\|_{2^*}^{2^*}\}.
\]
Given $u\in X_\Omega\setminus\{0\}$ there is exactly one $\lambda>0$ such that $\lambda u\in {\cal N}(\Omega)$, which defines the projection ${\cal T}:X_\Omega\setminus\{0\}\to {\cal N}(\Omega)$ as 
\[
{\cal T}(u)=\left(\frac{[u]_s^2}{\|u\|_{2^*}^{2^*}}\right)^{\frac{1}{2^*-2}}u.
\]
From the $1$-Lipschitzianity of the modulus we infer that $[|u|]_s\leq [u]_s$. Notice, however, that due to the nonlocality of the norm, for any (properly)  sign-changing $u\in {\cal N}(\Omega)$, it holds $|u|\notin {\cal N}(\Omega)$. However, a straightforward calculation shows that
\begin{equation}
\label{proptau}
I_\Omega\big({\cal T}(|u|)\big)\leq I_\Omega(u),\qquad \forall u\in {\cal N}(\Omega).
\end{equation}
The problem 
\begin{equation}
\label{problempositive}
\begin{cases}
(-\Delta)^s u=u^{\frac{N+2s}{N-2s}}&\text{in $\Omega$},\\
u=0&\text{in $\C\Omega$},\\
u\geq 0,\quad  u\neq 0,
\end{cases}
\end{equation}
 is equivalent to find critical points of $I_\Omega$ belonging to 
\[
{\cal N}_+(\Omega):=\{u\in X_\Omega: u\geq 0\}\cap {\cal N}(\Omega).
\]
When $\Omega=\R^N$, by \cite{classif}, the nonnegative, nontrivial critical points of $I$ on $\dot H^s(\R^N)$ are  exactly the functions
\begin{equation}
\label{Talentians}
U_{\eps, z}(x)=d_{N, s}\left(\frac{\eps}{\eps^2+|x-z|^2}\right)^{\frac{N-2s}{2}},
\end{equation}
for a suitable $d_{N, s}>0$. Let  
\begin{equation}
\label{defcinfty}
c_\infty:=\inf\big\{I(u):u\in {\cal N}_+(\R^N)\big\}.
\end{equation}
Using the fact that $[|u|]_s\leq [u]_s$ and considering ${\cal T}(|u|)$ for any $u\in {\cal N}(\R^N)$ we get
\begin{equation}
\label{defcinfty2}
c_\infty=\inf\big\{I(u):u\in {\cal N}(\R^N)\big\},
\end{equation}
and moreover it holds
\begin{equation}
\label{cinfty}
\begin{split}
c_\infty
&=\big(\frac{1}{2}-\frac{1}{2^*}\big)\inf\big\{[{\cal T}(v)]_s^2:v\geq 0, v\neq 0\big\}\\
&=\frac{s}{N}\Big(\inf\Big\{\frac{[v]_s^2}{\|v\|_{2^*}^2}: v\geq 0, v\neq 0\Big\}\Big)^{\frac{N}{2s}}\\
&=\frac{s}{N}S(N, s)^{\frac{N}{2s}}.
\end{split}
\end{equation}
Moreover, $u$ is a minimizer for \eqref{defcinfty} if and only if $u={\cal T}(v)$ for some minimizer of the problem
\[
\inf\big\{\frac{[v]_s^2}{\|v\|_{2^*}^2}: v\geq 0, v\neq 0\big\},
\]
therefore $u=U_{\eps, z}$ for some $\eps>0$ and $z\in \R^N$. 
Note that this implies
\begin{equation}
\label{energyT}
[U_{\eps, z}]_s^2=\|U_{\eps, z}\|_{2^*}^{2^*}=\frac{N}{s}c_\infty.
\end{equation}

We recall now some basic facts about the Nehari manifold setting we will work in. 

\begin{proposition}\label{Nehariproperties}  
The following facts holds.
\begin{enumerate}
\item
${\cal N}(\Omega)$ is a $C^2$ Hilbert manifold bounded away from $0$.
\item
For any $u_0\in {\cal N}(\Omega)$, $\nabla I_\Omega(u_0)=0$ if and only if $\nabla_{\cal N}I_\Omega(u_0)=0$ where $\nabla_{\cal N}I_\Omega(u_0)$ is the projection onto $T{\cal N}(u_0)$ of $\nabla I_\Omega(u_0)$ and $T{\cal N}(u_0)$ is the tangent space to ${\cal N}(\Omega)$ at $u_0$. In other words $u_0\in {\cal N}(\Omega)$ is a critical point of $I_\Omega: X_\Omega\to \R$ if and only if it is critical for $I_\Omega:{\cal N}(\Omega)\to \R$ as a functional on the Hilbert manifold ${\cal N}(\Omega)$.
\item
Given a bounded sequence $\{u_n\}\subseteq {\cal N}(\Omega)$, $\|\nabla_{\cal N}I_\Omega(u_n)\|\to 0$ if and only if $\|\nabla I_\Omega(u_n)\|\to 0$.
\end{enumerate}
\end{proposition}

\begin{proof}\ 
\vskip2pt
\noindent
{\em (1)}\ First observe that \eqref{energyT}, and the Sobolev inequality $\|u\|_{2^*}\leq C[u]_s$ imply that ${\cal N}(\Omega)$ is bounded away from zero, since 
\[
[u]_s^2=\|u\|_{2^*}^{2^*}\leq C^{2^*}[u]_s^{2^*} \quad \Rightarrow\quad [u]_s\geq C^{-\frac{2^*}{2^*-2}}.
\]
To prove that ${\cal N}(\Omega)$ is a Hilbert manifold write ${\cal N}(\Omega)=\{u\in X_\Omega\setminus\{0\}: N(u)=0\}$ where 
\[
N(u)=(\nabla I_\Omega(u), u)_s=[u]_s^2-\|u\|_{2^*}^{2^*}.
\]
Clearly $N\in C^2(X_\Omega)$ and for any $u\in {\cal N}(\Omega)$ it holds
\[
-(\nabla N(u), \frac{u}{[u]_s})_s=\frac{1}{[u]_s}(2^*\|u\|_{2^*}^{2^*}-2[u]_s^2)=(2^*-2)[u]_s\geq \eps>0
\]
being ${\cal N}(\Omega)$ bounded away from $0$. Therefore $\nabla N(u)\neq 0$ at any point $u\in {\cal N}(\Omega)$ which, through the implicit function theorem, completes the proof of the first assertion.
\vskip2pt
\noindent
{\em (2)}\  One implication is trivial, and we will prove the opposite one. Suppose $u_0\in {\cal N}(\Omega)$ is such that $\nabla_{\cal N}I_\Omega(u_0)=0$. By Riesz duality we will consider $\nabla_{\cal N} I_\Omega(u_0)$ as a vector belonging to 
\[
T{\cal N}(u_0)=(\nabla N(u_0))^\bot\subset X_\Omega,
\]
with the norm induced by $X_\Omega$.
We have that 
\[
 \nabla I_\Omega(u_0)=\nabla_{\cal N}I_\Omega(u_0)+\lambda \nabla N(u_0)
\]
for some $\lambda\in \R$, and taking the scalar product with $u_0$ we obtain, similarly as before,
\[
0=\lambda(\nabla N(u_0), u_0)_s=\lambda(2-2^*)[u_0]_s^2,
\]
which forces $\lambda=0$ and the claim.
\vskip2pt
\noindent
{\em (3)}\ The proof is analogous to the previous one. Since
\[
\nabla I_\Omega(u_n)=\nabla_{\cal N}I_\Omega(u_n)+\lambda_n\nabla N(u_n)
\]
and $\{\nabla N(u_n)\}$ is bounded being $\{u_n\}$ bounded, it suffices to show that $\lambda_n\to 0$. Taking the scalar product with $u_n$ we get
\[
|\lambda_n||(\nabla N(u_n), u_n)_s|\leq [\nabla_{\cal N} I_\Omega(u_n)]_s[u_n]_s
\]
and thus, being ${\cal N}(\Omega)$ bounded away from $0$ we obtain
\[
\eps |\lambda_n|\leq |\lambda_n|(2^*-2)[u_n]_s^2=|\lambda_n||(\nabla N(u_n), u_n)_s|\leq C[\nabla_{\cal N} I_\Omega(u_n)]_s\to 0.
\]
This concludes the proof.
\end{proof}

Recall that, given a topological space $A$, a subspace $B\subseteq A$ is called a {\em strong deformation retract} of $A$ if there exists ${\cal R}\in C^0([0,1]\times A, A)$ (a {\em homotopy retraction of $A$ on $B$}) such that  
\begin{enumerate}
\item
${\cal R}(0, x)= x$ for all $x\in A$,
\item
${\cal R}(t, x)=x$ for all $x\in B$, $t\in [0,1]$,
\item
${\cal R}(1, x)\in B$ for all $x\in A$.
\end{enumerate}
Given $J:M\to \R$, where $M$ is a $C^2$-Hilbert manifold, we denote by $C_{J}:=\{J(u): J'(u)=0, \ u\in M\}$ the set of critical levels. From \cite[Lemma 3.2]{Chang}, we get that the following deformation lemma holds true. 

\begin{proposition}\label{deformation}
Let $M$ be a $C^2$-Hilbert manifold and suppose $J\in C^{2}(M, \R)$ satisfies $(PS)_c$ for any $c\in [a, b]$. If $C_{J}\cap [a, b]=\emptyset$, then $\{u\in M:J(u)\leq a\}$ is a strong deformation retract of $\{u\in M: J(u)\leq b\}$.
\end{proposition}

\subsection{Nonexistence in the half-space}
Let us set $\R^N_+=\{x\in\R^N:x_N>0\}$. We have the following
regularity result.

\begin{lemma}
Any weak solution $u\in X_{\R^N_+}$ of 
\begin{equation}
\label{eqhalf}
\begin{cases}
(-\Delta)^su=|u|^{2^*-2}u&\text{in $\R^N_+$},\\
u\equiv 0&\text{in $\R^N\setminus \R^N_+$},
\end{cases}
\end{equation}
is bounded and continuous in $\R^N$.
\end{lemma}

\begin{proof}
The following is a modification of \cite[Theorem 3.2]{IMS}. Let us set $\gamma=(2^*/2)^{1/2}$ and $|t|_k:=\max\{|t|,k\}$, for any $k>0$. 
For all $r\geq 2$,  the mapping $t\mapsto t|t|_k^{r-2}$ is Lipschitz in $\R$, hence $u|u|_k^{r-2}$ belongs to $X_{\R^N_+}$. We test the weak form of \eqref{eqhalf} with $u|u|_k^{r-2}$, apply the fractional Sobolev inequality and the elementary inequality (see {\cite[Lemma 3.1]{IMS}})
\[
(a-b)(a|a|_k^{r-2}-b|b|_k^{r-2})\geq \frac{4(r-1)}{r^2}(a|a|_k^{\frac{r}{2}-1}-b|b|_k^{\frac{r}{2}-1})^2,
\]
to obtain
\begin{equation}\label{sc1}
\|u|u|_k^{\frac{r}{2}-1}\|_{2^*}^2\le S(N, s)^{-1}[u|u|_k^{\frac{r}{2}-1}]_s^2\le \frac{C r^2}{r-1} (u,u|u|_k^{r-2})_s\le Cr\int_{\R^N} |u|^{2^*}|u|_k^{r-2}\, dx,
\end{equation}
for some $C>0$ independent of $r\geq 2$ and $k>0$. Letting $k\to \infty$ and noting that $r^{1/r}$ is bounded for $r\geq 2$ gives
\begin{equation}
\label{sc2}
\|u\|_{\gamma^2 r}\leq C\Big(\int_{\R^N} |u|^{2^*+r-2}\, dx\Big)^{\frac 1 r},\qquad \gamma^2=\frac{2^*}{2}>1.
\end{equation}
 Let now $\bar r=2^*+1>2$, fix $\sigma>0$ such that $C\bar r\sigma<1/2$, where $C$ is the last constant appearing in \eqref{sc1} and $K_0$ so large that 
\begin{equation}
\label{piccolo}
\left(\int_{\{|u|> K_0\}}|u|^{2^*}\, dx\right)^{1-\frac{2}{2^*} }\leq\sigma.
\end{equation}
By H\"older inequality and \eqref{piccolo} we have
\begin{align*}
\int_{\R^N} |u|^{2^*}|u|_k^{\bar r-2} \, dx&\le K_0^{\bar r-2}\int_{|u|\leq K_0} |u|^{2^*}\, dx+\int_{\{|u|>K_0\}}|u|^{2^*}|u|_k^{\bar r-2}\, dx \\
&\le K_0^{\bar r-2}\|u\|_{2^*}^{2^*} +\Big(\int_{\R^N} (u^2|u|_k^{\bar r-2})^{\frac{2^*}{2}}\, dx\Big)^{\frac{2}{2^*}}\Big(\int_{\{\{|u|>K_0\}\}}|u|^{2^*} \, dx\Big)^{1-\frac{2}{2^*}} \\
&\le C(u)+\sigma\|u|u|_k^{\frac{\bar r}{2}-1}\|_{2^*}^2.
\end{align*}
Recalling that $C\bar r\sigma<1/2$, $\bar r=2^*+1$, inserting in \eqref{sc1}, and letting $k\to \infty$ we obtain
\[
\|u\|_{\bar q}\le \tilde C(u), \qquad \bar q=\frac{2^*(2^*+1)}{2}.
\]
Since $\bar q>2^*$, we can bootstrap a bound on higher $L^p$ norms through \eqref{sc2} starting from the $L^{\bar q}$ one. Define the sequence  
\[
p_0=\bar q ,\qquad p_{n+1}=\gamma^2(p_n+2 -2^*),
\]
which satisfies $p_n\to +\infty$ (since $\bar q$ is greater than the fixed point of $f(x)=\gamma^2(x+2-2^*)$). Now \eqref{sc2} reads
\[
\|u\|_{p_{n+1}}\leq  C \|u\|_{p_n}^{\gamma^2 \frac{p_n}{p_{n+1}}},
\]
which, iterated, gives $u\in L^{p_n}({\R^N})$ for any $n\geq 0$. For any $p\geq 2^*$ it holds
\[
\int_{\R^N} |u|^p\, dx\leq \int_{\{|u|\leq 1\}} |u|^{2^*}\, dx+ \int_{\{|u|\geq 1\}} |u|^p\, dx
\]
and since $p_n\to +\infty$ this implies that $u\in L^p(\R^N)$ for any $p\geq 2^*$. To obtain a uniform bound, 
we use H\"older's inequality on the last term in \eqref{sc2} with exponent $\gamma r/(r-1)>1$ to get
\[
\int_{\R^N} |u|^{r-1}|u|^{2^*-1}\, dx\leq \|u\|_{\gamma r}^{r-1} \Big(\int_{\R^N} |u|^{\frac{(2^*-1)\gamma r}{r(\gamma-1)+1}}\, dx\Big)^{1-\frac{r-1}{\gamma r}}
\]
and thus \eqref{sc2} becomes
\[
\|u\|_{\gamma^2 r}\leq C(u, r)\|u\|_{\gamma r}^{1-\frac 1 r},\qquad r\geq 2,\quad \gamma=\sqrt{\frac{2^*}{2}}>1,
\]
with 
\[
C(u, r)= C^{\frac 1 r} \Big(\int |u|^{\frac{(2^*-1)\gamma r}{r(\gamma-1)+1}}\, dx\Big)^{\frac{1}{r}-\frac{r-1}{\gamma r^2}}.
\]
We choose $r_n=\gamma^n\to +\infty$, letting
\[
t_n=\frac{(2^*-1)\gamma r_n}{r_n(\gamma-1)+1}\  \nearrow\  \bar p=\frac{(2^*-1)\gamma}{\gamma-1}>2^*.
\]
By monotone convergence theorem (separately on $\{|u|\leq 1\}$ and $\{|u|>1\}$) it holds
\[
\int_{\R^N} |u|^{t_n}\, dx\to \int_{\R^N} |u|^{\bar p}\, dx
\]
which is finite. In particular, $C(u, \gamma^n)$ is bounded for sufficiently large $n$ by a constant $C(u)$ and thus we obtained 
\[
\|u\|_{\gamma^{n+2}}\leq C(u) \|u\|_{\gamma^{n+1}}^{1-\frac{1}{\gamma^n}}
\]
for sufficiently large $n$. By a standard argument this implies that $\|u\|_\infty=\lim_n\|u\|_{\gamma^n}$ is finite. We now prove that $u\in C^0(\R^N)$. Interior regularity in $\{x_N>0\}$ follows from the local regularity result {\cite[Theorem 5.4]{IMS2}}, while from \cite[Theorem 4.4]{IMS2} we get 
\[
|u(x)|\leq C\|(-\Delta)^s u\|_\infty (x_N)_+^s=C\|u\|_\infty^{2^*-1}(x_N)_+^s,\qquad \forall x\in \R^N_+
\]
(notice that only the boundedness of $u$ and a uniform sphere condition on $\Omega$ is used in the proof of \cite[Theorem 4.4]{IMS2}). From this estimate we deduce that $u(x)\to 0$ as $x\to x_0\in \{x_N=0\}$, and thus the continuity of $u$ in the whole $\R^N$.
\end{proof}

\noindent
From \cite[Corollary 1.6]{MMW} we immediately obtain

\begin{corollary}
\label{nosol}
There is no nontrivial nonnegative weak solution $u\in X_{\R^N_+}$ of \eqref{eqhalf}.
\end{corollary}

\subsection{Global compactness}

We now recall the profile decomposition of the functional $I_\Omega$ proved in \cite{pal-pis2}, specialized to nonnegative Palais-Smale sequences in the manifold ${\cal N}_+(\Omega)$. 
 A Palais-Smale sequence for $I_\Omega:{\cal N}_+(\Omega)\to \R$ at the level $c\in \R$ is a sequence $\{u_n\}\subset {\cal N}_+(\Omega)$ such that 
\[
I_\Omega(u_n)\to c, \qquad \nabla_{\cal N} I_\Omega(u_n)\to 0. 
\]
Moreover $I_\Omega:{\cal N}_+(\Omega)\to \R$ is said to satisfy the Palais-Smale condition at level $c$ (briefly, $(PS)_c$) if every Palais-Smale sequence $\{u_n\}\subseteq {\cal N}_+(\Omega)$ at level $c$ is relatively compact.
We say that $c_0$ is a {\em critical level} for $I_\Omega$ and write $c_0\in C_{I_\Omega}$ if there exists $u_0\in {\cal N}_+(\Omega)$ such that $\nabla_{\cal N}I_\Omega(u_0)=0$ and $I_\Omega(u_0)=c_0$. 

\begin{proposition}
\label{PalaisSmale}
Let $\Omega$ be a bounded open subset of $\R^N$ with smooth boundary. Then
\begin{enumerate}
\item
$I_\Omega:{\cal N}_+(\Omega)\to \R$ satisfies $(PS)_c$ at every level $c$ of the form
\[
c\neq c_0+mc_\infty, \qquad c_0\in C_{I_\Omega}\cup \{0\}, \quad m\in \N.
\]
\item $I_\Omega:{\cal N}(\Omega)\to \R$ satisfies $(PS)_c$ at every level $c\in ]c_\infty, 2c_\infty[$.
\end{enumerate}
\end{proposition}

\begin{proof}
{\em (1)} - Let $\{u_n\}\subseteq {\cal N}_+(\Omega)$ be a Palais-Smale sequence at some level $c$. Proposition \ref{Nehariproperties} shows that $\{u_n\}$ is a Palais-Smale sequence for $I_\Omega:X_\Omega\to \R$. Suppose that $\{u_n\}$ is not relatively compact. Then by {\cite[Theorem 1.1]{pal-pis2}}  there exist $u^{(0)}$ solving problem \eqref{problem} such that $u_n\weakto u^{(0)}$ and $u^{(j)}\neq 0$, $V^{(j)}\in \R^N$, $j=1,\dots, m$ for some finite $m\in \N$  with the following properties:
\begin{equation}
\label{profiledec}
 c= I_\Omega(u^{(0)})+\sum_{j=1}^mI_\Omega(u^{(j)}),
 \end{equation}
\begin{equation}
\label{profiledec2}
\text{for $j=1,\dots , m$, $u^{(j)}$ solves \eqref{problem} in $\Omega^{(j)}:=\{x\in \R^N: V^{(j)}\cdot x> 0\}$ and}
\end{equation}
 \begin{equation}
 \label{profiledec3}
\text{$u^{(j)}$ is a weak limit in $\dot H^s(\R^N)$ of a subsequence of rescaled-translations of $u_n$}.
\end{equation}
Clearly $u_0\in {\cal N}_+(\Omega)\cup\{0\}$. On the other hand, $u_n$ is nonnegative for any $n$, and \eqref{profiledec3} implies  $u^{(j)}\geq 0$ for any $j=1,\dots, m$. By rotation invariance and the previous corollary, there are no solutions of \eqref{problempositive} in the half space, so that actually $V^{(j)}=0$ for all $j=1,\dots, m$, i.e. $u^{(j)}$ is a nonnegative, nontrivial, {\em entire} solution of \eqref{problempositive}. Since the latters are only of the form \eqref{Talentians}, we get $I(u^{(j)})=c_\infty$ for all $j=1, \dots, m$. Due to \eqref{profiledec} we thus obtain 
 \[
 c=I_\Omega(u_0)+mc_\infty,\qquad I_\Omega'(u_0)=0
 \]
 contrary to our assumption.
 \vskip2pt
 {\em (2)} - Let $\{u_n\}$ be a $(PS)_c$ sequence at a level $c\in \ ]c_\infty, 2c_\infty[$, and $\{u^{(j)}\}$ the corresponding profile decomposition. Due to \cite[Lemma 2.5]{SSS} any sign-changing solution $u^{(j)}$ of \eqref{problem} in an arbitrary domain $\Omega$ satisfies $I_\Omega(u^{(j)})\geq 2c_\infty$. From \eqref{profiledec} we infer from $c<2c_\infty$ that no $u^{(j)}$ is sign changing, and from $c>c_\infty$ that $m=0$. The compactness now follows from {\cite[Theorem 1.1]{pal-pis2}}. 
 \end{proof}

\section{Estimates}
\label{Estimates}

\noindent
Let $R$, $\rho>0$. Choose $\psi\in C^\infty_c(\R^N)$ such that 
\[
0\leq \psi\leq 1, \qquad \psi(x)=0\quad \text{if $|x|\geq 2\rho$}, \qquad  \psi(x)= 1\quad \text{if $|x|\leq \rho$},
\]
and $\omega\in C^\infty(\R^{N-1})$ such that 
\[
0\leq \omega\leq 1,\qquad \omega(x')=1\quad \text{if $|x'|\geq 2$}, \qquad \omega(x')=0\quad \text{if $|x'|\leq 1$}.
\]
Finally, for any $\delta>0$ and  $z\in {\mathbb S}_R^{N-1}$, define
\[
\omega_\delta(x)=\omega(\frac{x'}{\delta}), \qquad \u (x)=\omega_\delta(x)\psi(x-z)\U(x).
\]

\begin{proposition}
\label{one}
There exists $C_1$ such that for each $\eps>\delta>0$ sufficiently small and $z\in \mathbb{S}_R^{N-1}$ it holds
\begin{equation} \label{stima1}
[\u]_s^2 \leq \frac{N}{s}c_\infty+C_1\frac{\delta^{N-1-2s}}{\eps^{N-2s}}+o(1),
\end{equation}
where $o(1)\to 0$ for $\eps \to 0$ independently of $\delta$.
\end{proposition}
\begin{proof}
In the following by $C$ we denote a generic constant depending only on $\psi$, $\omega$, $R$, $\rho$ and the numerical data $s, N$.

We let $\eta(x)=\omega_\delta(x)\psi(x-z)$ and, being $\u=\eta \U\in C^\infty_c(\R^N)$, notice that
\[
[\u]_s^2=\int_{\R^N}(-\Delta)^s\u(x) \u(x) dx,
\] 
therefore 
\[
\begin{split}
[\u]_s^2&=C(N, s)\int_{\R^{N}}\eta(x)\U(x){\rm {\small PV}}\int_{\R^N}\frac{\eta(x) \U (x)-\eta(y) \U(y)}{|x-y|^{N+2s}}\, dy\, dx\\
&=C(N, s)\int_{\R^{N}}\eta^2(x)\U(x){\rm PV}\int_{\R^N}\frac{\U(x)-\U(y)}{|x-y|^{N+2s}} \, dy\, dx\\
&\quad +C(N, s)\int_{\R^{N}}\eta(x)\U(x){\rm PV}\int_{\R^N}\frac{(\eta(x)-\eta(y))}{|x-y|^{N+2s}}\U(y)\, dy\, dx\\
&=\int_{\R^{N}}\eta^2\U(-\Delta)^s\U dx+C\int_{\R^{N}}\eta(x)\U(x){\rm PV}\int_{\R^N}\frac{\eta(x)-\eta(y)}{|x-y|^{N+2s}}\U(y)\, dy\, dx\\
& =I_1+CI_2.
\end{split}
\]
We estimate separately the two integrals.
For $I_1$ we have $0\leq \eta\leq 1$ and $\U(-\Delta)^s\U=\U^{2^*}$, thus
\[
I_1=\int_{\R^N}\U^{2^*}\eta^2\, dx\leq \|\U\|_{2^*}^{2^*}=\frac{N}{s}c_\infty
\]
by \eqref{energyT}. For $I_2$ notice that
\[
\begin{split}
2I_2&=\int_{\R^{N}}\eta(x)\U(x){\rm PV}\int_{\R^N}\frac{\eta(x)-\eta(y)}{|x-y|^{N+2s}}\U(y)\, dx\, dy\\
&\quad+\int_{\R^{N}}\eta(y)\U(y){\rm PV}\int_{\R^N}\frac{\eta(y)-\eta(x)}{|x-y|^{N+2s}}\U(x)\, dx\, dy\\
&=\int_{\R^{2N}}\frac{(\eta(x)-\eta(y))^2}{|x-y|^{N+2s}}\U(x)\U(y)\, dx\, dy.
\end{split}
\]
Being $|\omega_\delta|\leq 1$ and $\eta(x)=\omega_\delta(x)\psi(x-z)$, we have
\[
\begin{split}
|\eta(x)-\eta(y)|&\leq |\psi(x-z)||\omega_\delta(x)-\omega_\delta(y)|+|\omega_\delta(y)||\psi(x-z)-\psi(y-z)|\\
&\leq \psi(x-z)|\omega_\delta(x)-\omega_\delta(y)|+|\psi(x-z)-\psi(y-z)|.
\end{split}
\] 
Therefore we get, through a translation
\begin{equation}
\label{temp}
\begin{split}
I_2&\leq 2\int_{\R^{2N}}\frac{(\omega_\delta(x)-\omega_\delta(y))^2}{|x-y|^{N+2s}}\psi^2(x-z)\U(x)\U(y)\, dx\, dy\\
&\qquad+2\int_{\R^{2N}}\frac{(\psi(x)-\psi(y))^2U_{\eps, 0}(x)U_{\eps, 0}(y)}{|x-y|^{N+2s}} dx\, dy.
\end{split}
\end{equation}
To estimate the first term, let $h\in C_c^{\infty}(\R)$ be such that
$\psi(x', x_N)\leq h(x_N)\leq 1$, and compute 
\begin{equation}
\label{omegadelta}
\begin{split}
\int_{\R^{2N}}&\frac{(\omega_\delta(x)-\omega_\delta(y))^2}{|x-y|^{N+2s}}\psi^2(x-z)\U(x)\U(y)\, dx\, dy\\
&\leq\frac{1}{\eps^{N-2s}}
\int_{\R^{2N}}\frac{(\omega_\delta(x')-\omega_\delta(y'))^2}{(|x'-y'|^2+|x_N-y_N|^2)^{\frac{N+2s}{2}}}h^2(x_N-z_N)\, dx\, dy\\
&= \frac{1}{\eps^{N-2s}}
\int_{\R^{2(N-1)}}\frac{(\omega_\delta(x')-\omega_\delta(y'))^2}{|x'-y'|^{N+2s}}\int_{\R^2}\frac{h^2(x_N-z_N)}{(1+\frac{|x_N-y_N|^2}{|x'-y'|^2})^{\frac{N+2s}{2}}}\, dx_N\, dy_N\, dx'\, dy'\\
&=\frac{1}{\eps^{N-2s}}
\int_{\R^{2(N-1)}}\frac{(\omega_\delta(x')-\omega_\delta(y'))^2}{|x'-y'|^{N-1+2s}}\, dx'\, dy'\int_{\R}h^2(x_N-z_N)\, dx_N\int_\R\frac{1}{(1+t^2)^{\frac{N+2s}{2}}}\, dt\\
& \leq \frac{C}{\eps^{N-2s}}\delta^{N-1-2s}\int_{\R^{2(N-1)}}\frac{(\omega(x')-\omega(y'))^2}{|x'-y'|^{N-1+2s}}\, dx'\, dy'=C\frac{\delta^{N-1-2s}}{\eps^{N-2s}}.
\end{split}
\end{equation}
Finally, we estimate the second term in \eqref{temp}.  Notice that by scaling
\[
\begin{split}
\int_{\R^{2N}}\frac{(\psi(x)-\psi(y))^2U_{\eps, 0}(x)U_{\eps, 0}(y)}{|x-y|^{N+2s}} dx\, dy&=
\int_{\R^{2N}}\frac{(\psi(\eps x)-\psi(\eps y))^2\Uoz( x)\Uoz( y)}{| x- y|^{N+2s}} d x\, dy\\
&\leq {\rm Lip}(\psi)^2\eps^2\int_{\R^{2N}}\frac{\Uoz( x)\Uoz( y)}{|x- y|^{N-2+2s}} dx\, dy.
\end{split}
\]
We apply  Hardy-Littlewood-Sobolev's inequality to the last integral with exponents 
\[
\frac{1}{p}+\frac{1}{p}=1+\frac{2-2s}{N} \quad \leftrightarrow\quad p=\frac{2N}{N-2s+2}
\]
and obtain
\[
\int_{\R^{2N}}\frac{\Uoz(x)\Uoz( y)}{| x- y|^{N-2+2s}} dx\, dy\leq C\|\Uoz\|_p^2
\]
which is finite as long as $p(N-2s)>N$, i.e., $N>2+2s$. This concludes the proof for $N\geq 4$. If $N=2$ or $3$ we write
\[
\begin{split}
\int_{\R^{2N}}&\frac{(\psi(x)-\psi(y))^2U_{\eps, 0}(x)U_{\eps, 0}(y)}{|x-y|^{N+2s}} dx\, dy\\
&\qquad \leq {\rm Lip}(\psi)\int_{B_{4\rho}\times B_{4\rho}}\frac{U_{\eps, 0}(x)U_{\eps, 0}(y)}{|x-y|^{N-2+2s}} dx\, dy+2\int_{B_{2\rho}\times \C B_{4\rho}}\frac{U_{\eps, 0}(x)U_{\eps, 0}(y)}{|x-y|^{N+2s}} dx\, dy=I_3+I_4
\end{split}
\]
since if $x\in B_{4\rho}\setminus B_{2\rho}$ and $y\in \C B_{4\rho}$, $\psi(x)=\psi(y)=0$. The integral $I_3$ can be estimated through the Hardy-Littlewood-Sobolev inequality with exponent $p$ given by
\[
\frac{1}{2^*}+\frac{1}{p}=1+\frac{2-2s}{N} \quad \leftrightarrow\quad p=\frac{2N}{N-2s+4}, 
\]
and we obtain
\[
I_3\leq C\|U_{\eps, 0}\|_{2^*}\left(\int_{B_{4\rho}}\left(\frac{\eps}{\eps^2+|x|^2}\right)^{\frac{N-2s}{2}p}\, dx\right)^{\frac 1 p}\leq C\eps^{\frac{N-2s}{2}}\left(\int_{B_{4\rho}}\frac{1}{|x|^{p(N-2s)}}\, dx\right)^{\frac 1 p}
\]
the last integral being finite as long as $p(N-2s)<N$. Substituting $p$, we get $N<2s+4$, which holds for $N=2,3$. For $I_4$ we directly have 
\[
U_{\eps, 0}(y)\leq C\eps^{\frac{N-2s}{2}},\quad \forall y\in \C B_{4\rho}\quad \text{ and }\quad U_{\eps, 0}(x)\leq C\left(\frac{\eps}{|x|^2}\right)^{\frac{N-2s}{2}},\quad \forall x\in B_{2\rho}
\]
which implies, being $|z|=|x-y|\geq 2\rho$ for  any $x\in B_{2\rho}$, $y\in \C B_{4\rho}$,
\[
I_4\leq  C\eps^{N-2s}\int_{B_{2\rho}}\frac{1}{|x|^{N-2s}}\, dx\int_{\{|z|\geq 2\rho\}}\frac{1}{|z|^{N+2s}}\, dz\leq C\eps^{N-2s}.
\]

\end{proof}

\begin{proposition}
\label{two}
There exists $C_2>0$ such that for $\eps>\delta>0$ sufficiently small and $|z|=R$
\begin{equation}\label{stima3}
\int_{\R^N} \u^{2^*}\,dx\geq
\frac{N}{s}c_\infty-C_2\frac{\delta^{N-1}}{\varepsilon^N} -o(1),
\end{equation}
where $o(1)\to 0$ for $\eps \to 0$ independently of $\delta$.
\end{proposition}
\begin{proof}
Since by \eqref{energyT}
\[
\U^{2^*}(x)=d_{N, s}^{2^*}\left(\frac{\eps}{\eps^2+|x-z|^2}\right)^N,\qquad \int_{\R^N} \U^{2^*}\, dx=\frac{N}{s}c_\infty,
\]
we have
\[
\begin{split}
\int_{\R^N} \U^{2^*}\,dx-\int_{\R^N} \u^{2^*}\,dx&\leq
\int_{B_{\rho}(z)} \U^{2^*}(1-\omega_\delta)\, dx+\int_{\C B_{\rho}}\U^{2^*}\, dx\\
&\leq\frac{C}{\eps^N}|B_{\rho}(z)\cap \{\omega_\delta<1\}|+C\eps^N\int_{\C B_\rho(z)}\frac{1}{|x-z|^{2N}}\, dx\\
&\leq C\frac{\delta^{N-1}}{\eps^N}+C\eps^N.
\end{split}
\]
This concludes the proof.
\end{proof}

Finally, we show how to modify the previous proofs to cater with the borderline case $N=2$, $s=1/2$.

\begin{lemma}
For any $\theta\in\ ]0,1[$ there exist $R_\theta>1$ and a function $\eta_\theta\in C^\infty_c(\R)$ such that
\begin{equation}
\label{propetazero}
\eta_\theta(x)=1\quad \text{if $|x|\leq 1$}, \qquad \eta_\theta(x) =0\quad \text{in $|x|\geq R_\theta$}, \qquad 0\leq \eta_\theta\leq 1
\end{equation}
and
\begin{equation}
\label{propeta}
[\eta_\theta]_{1/2}^2\leq \frac{C}{|\log\theta|}.
\end{equation}
\end{lemma}

\begin{proof}
It follows from the property (see \cite[Theorem 2.6.14]{Ziemer}) of the Bessel capacity of intervals
\begin{equation}
\label{capacity}
B_{1/2, 2}([-\theta, \theta])=\inf\{\|u\|_2^2: u\geq 0, \text{ $g_{1/2}* u\geq 1$ on $[-\theta, \theta]$}\}\leq \frac{C}{|\log\theta|},
\end{equation}
where $g_{1/2}$ is the Bessel potential in $\R$ (so that ${\cal F}(g_{1/2})(\xi)=(2\pi)^{-1/2}(1+|\xi|^2)^{-1/4}$). Recall (see \cite[Proposition 4, sec. 3.5]{Stein}) that $\eta \in H^{1/2}(\R)$ if and only if $\eta=g_{1/2}* u$ for some $u\in L^2(\R)$. The density of $C^\infty_c(\R)$ in $H^{1/2}(\R)$, the lattice property of the latter and \eqref{Ftransform}  imply
\[
\begin{split}
\inf\{[\eta]^2_{1/2}:\eta\in C^\infty_c(\R), \eta\geq \chi_{[-\theta, \theta]}\}&=
\inf\{[\eta]^2_{1/2}:\eta\in H^{1/2}(\R), \text{$\eta\geq 1$ on $[-\theta, \theta]$}\}\\
&=\inf\{\int_{\R^N} |\xi||\cal F(\eta)|^2\, d\xi: \eta\in H^{1/2}(\R), \text{$\eta\geq 1$ on $[-\theta, \theta]$}\}\\
&\leq 
\inf\{\int_{\R^N}(1+|\xi|^2)^{1/2}|\cal F(\eta)|^2\, d\xi: \eta\in H^{1/2}(\R), \text{$\eta\geq 1$ on $[-\theta, \theta]$}\}\\
&\leq C\inf\{\|u\|_2^2: \text{$g_{1/2}* u\geq 1$ on $[-\theta, \theta]$}\}\leq CB_{1/2, 2}([-\theta, \theta]),
\end{split}
\]
which together with \eqref{capacity} gives the claim.
\end{proof}

Let us define for any $1>\theta>\lambda>0$,  the function $\omega_{\theta, \lambda}\in C^\infty(\R)$
\[
\omega_{\theta, \lambda}(x_1)=1-\eta_\theta(\frac{x_1}{\lambda}),
\]
and for $\rho, R>0$ and $z\in {\mathbb S}^{1}_R$, $x=(x_1, x_2)\in \R^2$ we define, similarly to the beginning of the section
\[
u_{\theta, \lambda, \eps, z}(x)= \omega_{\theta, \lambda}(x_1)\psi(x-z)\U(x).
\]

\begin{proposition}
\label{propbis}
Let $N=2$, $s=1/2$. There exists $C_1$ such that if $1>\eps>\theta>\lambda>0$ and $z\in \mathbb{S}^{N-1}_R$ it holds
\begin{equation}
\label{propbisone}
[u_{\theta, \lambda, \eps, z}]_{s}^2\leq \frac{N}{s}c_\infty +\frac{C_1}{|\log\theta|\eps^{N-2s}}+o(1),
\end{equation}
\begin{equation}
\label{propbistwo}
\|u_{\theta,\lambda,\eps, z}\|_{2^*}^{2^*}\geq \frac{N}{s}c_\infty -C_1\frac{\lambda R_\theta}{\eps^N}-o(1),
\end{equation}
where $o(1)\to 0$ for $\eps \to 0$ independently of $\theta$ and $\lambda$.
\end{proposition}

\begin{proof}
Regarding \eqref{propbisone} we can repeat the proof of Proposition \ref{one}. Since the only thing we are changing is the use of $\omega_{\theta, \lambda}$ instead of $\omega_\delta$, it suffices to focus on the last inequality in \eqref{omegadelta}, where in this case $N-1+2s=2$. By scaling
\[
\int_{\R^2}\frac{(\omega_{\theta, \lambda}(x_1)-\omega_{\theta, \lambda}(y_1))^2}{|x_1-y_1|^2}\, dx_1\, dy_1=\int_{\R^2}\frac{(\omega_{\theta,1}(x_1)-\omega_{\theta, 1}(y_1))^2}{|x_1-y_1|^{2}}\, dx_1\, dy_1=[\eta_\theta]^2_{1/2}
\]
and \eqref{propeta} gives \eqref{propbisone}. To obtain \eqref{propbistwo}, we use scaling and \eqref{propetazero} to get 
\[
|B_\rho(z)\cap \{\omega_{\theta, \lambda}<1\}|\leq C\lambda R_\theta,
\]
and proceed as in the proof of Proposition \ref{two}.
\end{proof}

\section{Existence}
\label{Existence}
\noindent
In the following we shall assume that
\begin{equation}
\label{assumption}
\text{there is no critical point for $I_\Omega$ on ${\cal N}_+(\Omega)$ at a level $c\in [c_\infty, 2c_\infty]$,}
\end{equation}
and that $\Omega\subseteq B_{R_3}\setminus B_{R_0}$. For any $u\in L^{2^*}(\R^N)\setminus \{0\}$ we define its barycenter as
\[
\beta(u)=\frac{\int_{B_{R_3}} x |u|^{2^*}\, dx}{\int_{\R^N} |u|^{2^*}\, dx}.
\]
Clearly $\beta:L^{2^*}(\R^N)\setminus\{0\}\to \R^N$ is a continuous function w.r.t. the strong topology, and as long as $\beta(u)\neq 0$ we can define
\[
\bar\beta (u)=\frac{\beta (u)}{|\beta(u)|}.
\]

\begin{lemma}
\label{beta}
There exists $\eps_0=\eps_0(N, s, R_0, R_3)>0$  such that for any $\Omega\subseteq B_{R_3}\setminus B_{R_0}$
\[
I_\Omega(u)\leq c_\infty+\eps_0\quad u\in {\cal N}(\Omega)\quad  \Rightarrow \quad |\beta (u)|\geq \frac{R_0}{2}.
\]
\end{lemma}

\begin{proof}
Suppose not and let $A=\{x\in \R^N: R_0<|x|<R_3\}$. Then there exists a sequence $\{\Omega_n\}$ such that $\Omega_n\subseteq A$ and $\{u_n\}\subset {\cal N}(\Omega_n)$ such that $I_{\Omega_n}(u_n)\leq c_\infty+\frac{1}{n^2}$ and $|\beta(u_n)|<R_0/2$. Since ${\cal N}(\Omega_n)\subseteq {\cal N}(A)$, by Ekeland's Variational principle \cite[Proposition 5.1]{Ekeland}, we can pick a sequence $\{v_n\}\subseteq {\cal N}(A)$ such that 
\[
[v_n-u_n]_s\leq \frac 1 n, \qquad I_A(v_n)\leq c_\infty+\frac 1 n, \qquad [\nabla_{\cal N}I_A(v_n)]_s\leq \frac 1 n,
\]
where the norms are taken in $X_A$.
By proposition \ref{Nehariproperties}, $\{v_n\}$ is a PS sequence for $I_A:X_A\to \R$ at level $c_\infty$, thus by \cite[Theorem 1.1]{pal-pis2} the profile decomposition \eqref{profiledec}--\eqref{profiledec3} holds true for some $v^{(0)}\in X_A$, $v^{(j)}\in X_{\Omega^{(j)}}$, $j=1,\dots, m$. Since $c_\infty$ is not a critical level, $v^{(0)}=0$, and by \cite[Lemma 2.5]{SSS} no $v^{(j)}$ can be sign-changing. Using also Corollary \ref{nosol} we obtain that $m=1$ and $v^{(1)}= U_{\eps, z}$ for some $z\in \R^N$, $\eps>0$. Therefore \cite[Theorem 1.1, (1.6)]{pal-pis2} ensures that there exist $\eps_n>0$, $z_n\in A$ such that 
\[
[v_n-U_{\eps_n, z_n}]_s\to 0 \quad \text{in $\dot H^s(\R^N)$},
\]
where $\eps_n\to 0$ (since $U_{\eps, z}\notin X_A$). Suppose, without loss of generality, that $z_n\to z\in \bar A$. By scaling and \eqref{energyT}
\[
U_{\eps_n, z_n}^{2^*}\weakto \frac N s c_\infty \delta_z\qquad \text{as $\eps_n\to 0$},
\]
in the sense of measures. We claim that $\beta(v_n)\to z\in \bar A$ as $n\to +\infty$: since it holds $\|v_n-U_{\eps_n, z_n}\|_{2^*}\to 0$  by Sobolev embedding, and  $\|U_{\eps_n, z_n}\|_{2^*}^{2^*}\equiv\frac N s c_\infty$ this follows from
\[
\int_{B_{R_3}}x |v_n|^{2^*}\, dx =\int_{B_{R_3}} x(|v_n|^{2^*}-U_{\eps_n, z_n}^{2^*})\, dx+\int_{B_{R_3}} x U_{\eps_n, z_n}^{2^*}\, dx\to \frac N s c_\infty z.
\]
 However from $\|u_n-v_n\|_{2^*}\to 0$ we deduce $|\beta(v_n)-\beta(u_n)|\to 0$ and so, by our assumption, $|z|\leq R_0/2$, which  is a contradiction with $z\in \bar A$.
\end{proof}

\begin{lemma}
\label{varphi}
Let $N\geq 3$ and $s\in \ ]0,1[$ or $N=2$ and $s\in \ ]0, \frac1 2]$. For any $\bar \eps>0$ , there exists $\delta(\bar\eps, N, s, R_1, R_2)>0$ such that to any $\Omega\subseteq B_{R_3}$ satisfying \eqref{thin-tunnel2} there corresponds $\varphi:\mathbb{S}^{N-1}\to {\cal N}_+(\Omega)$ with the following properties:
\begin{equation}
\label{energy}
I_\Omega(\varphi(x))\leq c_\infty+\bar\eps \qquad \forall x\in \mathbb{S}^{N-1};
\end{equation}
\begin{equation}
\label{degree}
0\notin \beta(\varphi(\mathbb{S}^{N-1})),\qquad |\bar\beta(\varphi(x))-x|\leq 1\qquad \forall x\in \mathbb{S}^{N-1}.
\end{equation}
\end{lemma}

\begin{proof}
First we let, from \eqref{thin-tunnel}, $R=(R_1+R_2)/2$ and 
\[
\rho<\min\Big\{\frac{R}{10}, \frac{R_2-R_1}{2}\Big\}.
\]
Consider first the case $N-1-2s>0$ (i.e. $N\geq 3$ and $s\in \ ]0,1[$ or $N=2$ and $s\in \ ]0,1/2[$). For $z\in \mathbb{S}^{N-1}_R$ the functions $\u$ constructed in the previous section belong to $X_\Omega$, as soon as $\Omega$ satisfies \eqref{thin-tunnel2}. Without loss of generality, we can assume $1\gg \eps>0$ and set
\[
\delta=\eps^\alpha,\qquad \text{for some }\quad\alpha>\frac{N-2s}{N-1-2s}>\frac{N}{N-1}>1.
\]
For such a choice, \eqref{stima1} and \eqref{stima3} read
\begin{equation}
\label{estimatesenergy}
[u_{\delta, \eps, z}]_s^2\leq \frac{N}{s}c_\infty+o(1),\qquad \|u_{\delta, \eps, z}\|_{2^*}^{2^*}\geq \frac{N}{s}c_\infty-o(1),\qquad \forall z\in \mathbb{S}^{N-1}_R.
\end{equation}
In the case $N=2$, $s=1/2$ we instead use Proposition \ref{propbis}. First we choose $\theta=e^{-\eps^{-\alpha}}$, $\alpha>1$ and then $\lambda>0$ such that $\lambda=\eps^{1+\alpha}/R_\theta$. Then \eqref{propbisone} and \eqref{propbistwo} provide \eqref{estimatesenergy} for $u_{\theta, \lambda, \eps, z}$.
Let us call, for $\delta, \theta, \lambda$ depending on $\eps$ as before,
\[
u_{\eps, z}=
\begin{cases}
u_{\delta, \eps, z} &\text{if $N-1-2s>0$},\\
u_{\theta, \lambda, \eps, z}&\text{if $N=2$ and $s=1/2$},
\end{cases}
\qquad 
\delta=
\begin{cases}
\eps^\alpha&\text{if $N-1-2s>0$},\\
\lambda\theta&\text{if $N=2$, $s=1/2$},
\end{cases}
\]
and define, for any  $x\in \mathbb{S}^{N-1}$
\[
\varphi(x)={\cal T}(u_{\eps, Rx})\in {\cal N}_+(\Omega).
\]
Since 
\[
I_\Omega(\varphi(x))=\frac{s}{N}\left(\frac{[u_{\eps, Rx}]_s^2}{\|u_{\eps, Rx}\|_{2^*}^2}\right)^{\frac{N}{2s}}
\leq \frac{s}{N}\left(\frac{\frac{N}{s}c_\infty+o(1)}{(\frac{N}{s}c_\infty-o(1))^{2/2^*}}\right)^{\frac{N}{2s}}= c_\infty+o(1)
\]
we have that \eqref{energy} holds for sufficiently small $\eps$ (and thus $\delta$). To prove \eqref{degree} observe that, for any $z\in \mathbb{S}^{N-1}_R$, $u_{\eps, z}$ is supported in $B_{2\rho}(z)$, therefore its barycenter lies in $B_{2\rho}(z)$, and in particular is nonzero, being $2\rho<R$. Since $\beta({\mathcal T}(u_{\eps, z}))=\beta(u_{\eps, z} )$, it holds 
\[
|\beta(\varphi(x))-Rx|\leq 2\rho,
\]
which implies 
\[
\begin{split}
\left|\frac{\beta(\varphi(x))}{|\beta(\varphi(x))|}-x\right|&\leq \left|\frac{\beta(\varphi(x))}{|\beta(\varphi(x))|}-\frac{\beta(\varphi(x))}{R}\right|
+\left|\frac{\beta(\varphi(x))}{R}-x\right|\\
&\leq |\beta(\varphi(x))|\left|\frac{1}{R}-\frac{1}{|\beta(\varphi(x))|}\right|+\frac{2\rho}{R}\\
&\leq (R+2\rho)\frac{2\rho}{R(R-2\rho)}+\frac{2\rho}{R}<1
\end{split}
\]
being $10\rho<R$.
\end{proof}
 
We now define the minimax problem providing the critical level for $I_\Omega$. In the following we will assume that $\delta$ is small enough so that Lemma \ref{varphi} holds, and $\Omega$ satisfies \eqref{thin-tunnel2} for such a $\delta$.
\vskip3pt
\noindent
Let us set 
\[
\Gamma:=\{\gamma\in C^0(\mathbb{S}^{N-1}, {\cal N}_+(\Omega): \beta(\gamma(x))\neq 0 \text{ for all $x\in \mathbb{S}^{N-1}$ and ${\rm deg} (\bar\beta\circ \gamma, \mathbb{S}^{N-1})\neq 0$}\}.
\]
Observe that $\bar \beta\circ \varphi\in \Gamma$ since \eqref{degree} implies that
\[
H(t, x)=\frac{t\bar\beta(\varphi(x))+(1-t)x}{|t\bar\beta(\varphi(x))+(1-t)x|}
\]
is a homotopy between $\bar\beta\circ\varphi$ and the identity map of $\mathbb{S}^{N-1}$. Thus by the homotopy invariance of the degree we get ${\rm deg} (\bar\beta\circ \varphi, \mathbb{S}^{N-1})={\rm deg} ({\rm Id}, \mathbb{S}^{N-1})=1$. Therefore the minimax problem 
\[
c_1:=\inf_{\gamma\in \Gamma}\sup_{x\in \mathbb{S}^{N-1}}I_\Omega(\gamma(x)),
\]
is well defined. Let furthermore
\[
\bar c:=\inf\{I_\Omega(u): u\in {\cal N}_+(\Omega), \beta(u)\neq 0, \bar\beta(u)=e_N\},
\]
where $e_N=(0, \dots, 0, 1)$.

\begin{lemma}
\label{lemmapro} Let $N\geq 3$ and $s\in \ ]0,1[$ or $N=2$ and $s\in \ ]0, \frac1 2]$. For any $\bar\eps>0$, there exists $\delta(\bar\eps, N, s, R_1, R_2)>0$ such that to any bounded $\Omega\subseteq B_{R_3}$ satisfying \eqref{thin-tunnel}, \eqref{thin-tunnel2}, it holds
\[
c_\infty<\bar c\leq c_1\leq c_\infty+\bar \eps
\]
\end{lemma}

\begin{proof}
We fix $\delta>0$ so that Lemma \eqref{varphi} holds for $\bar \eps$, providing the corresponding $\varphi$. First observe that since ${\rm deg} (\bar\beta\circ \varphi, \mathbb{S}^{N-1})=1\neq 0$ there is $x\in \mathbb{S}^{N-1}$ such that $\bar\beta(\varphi(x))=e_N$. Therefore $\bar c$ is well defined as well. By the same reason, given any $\gamma\in \Gamma$, there exists $x_\gamma$ such that $\bar\beta(\gamma(x_\gamma))=e_N$, so that 
\[
\bar c\leq\inf_{\gamma\in \Gamma} I_\Omega(\gamma(x_\gamma))\leq c_1.
\]
Since, as noted before, $\varphi\in \Gamma$, we have through \eqref{energy}
\[
c_1\leq \sup_{x\in \mathbb{S}^{N-1}}I_\Omega(\varphi(x))\leq c_\infty+\bar\eps.
\]
The argument which shows that $c_\infty<\bar c$ relies on \eqref{thin-tunnel} and is analogous to the proof of Lemma \ref{beta}.
\end{proof} 

\begin{theorem}
Let $N\geq 3$ and $s\in \ ]0,1[$ or $N=2$ and $s\in \ ]0,\frac 1 2]$. Then there exists $\delta>0$ such that if $\Omega\subseteq B_{R_3}\setminus B_{R_0}$ is a smooth open set satisfying \eqref{thin-tunnel}, \eqref{thin-tunnel2}, $I_\Omega$ has a critical point in $X_\Omega$ at some level $c\in \ ]c_\infty, 2c_\infty[$.
\end{theorem}

\begin{proof}
Let $\eps_0=\eps(N, s, R_0, R_3)\in \ ]0, c_\infty[$ be given in Lemma \ref{beta} and let $\delta>0$ be such that Lemma \ref{varphi} and Lemma \ref{lemmapro} hold for $\bar\eps=\eps_0/2$. Finally choose $a, b\in \ ]c_\infty, 2c_\infty[$ such that
\[
c_\infty<a<\bar c\leq c_1<b<c_\infty+\eps_0<2c_\infty.
\]
 By Proposition \ref{PalaisSmale}, {\em (2)}, $I_\Omega$ satisfies $(PS)_c$ on ${\cal N}(\Omega)$ for all $c\in [a, b]$. Applying Proposition \ref{deformation}, fix a homotopy retraction ${\cal R}$ of $\{u\in {\cal N}(\Omega):I_\Omega(u)\leq b\}$ on $\{u\in {\cal N}(\Omega):I_\Omega(u)\leq a\}$ and pick $\bar\gamma\in \Gamma$ such that 
\[
\sup_{x\in \mathbb{S}^{N-1}}I_\Omega(\bar \gamma(x))\leq b.
\]
We claim that 
\[
\bar \gamma_1:= {\cal T}(|{\cal R}(1, \bar\gamma)|)\in \Gamma.
\]
Indeed, for any $x\in \mathbb{S}^{N-1}$, $\bar \gamma_1(x)\in {\cal N}_+(\Omega)$ and \eqref{proptau} ensures 
\[
I(\bar\gamma_1(x))\leq I\big({\cal R}(1,\bar \gamma(x))\big)\leq a<c_\infty+\eps_0.
\]
 By Lemma \ref{beta} it holds $\beta(\bar\gamma_1(x))\neq 0$ for all $x\in \mathbb{S}^{N-1}$, while we claim 
 \[
 (x, t)\mapsto H(t, x):=\bar\beta\big({\cal T}\big(|{\cal R}(t, \bar\gamma(x))|\big)\big)
 \]
 defines a homotopy in $\mathbb{S}^{N-1}$ between $\bar\beta\circ\bar\gamma$ and $\bar\beta\circ  \bar\gamma_1$. Indeed using \eqref{proptau}, $I({\cal R}(t, x))\leq c_\infty+\eps_0$ and Lemma \ref{beta} we obtain that $\beta\big({\cal T}\big(|\cal{R}(t, \bar\gamma(x))|\big)\big)\neq 0$ for all $(t, x)\in [0,1]\times \mathbb{S}^{N-1}$, and thus $H$ is continuous. This ensures that ${\rm deg}(\bar\beta\circ \bar\gamma_1, \mathbb{S}^{N-1})={\rm deg}(\bar\beta\circ \bar\gamma, \mathbb{S}^{N-1})\neq 0$. We thus reached a contradiction, being 
\[
c_1=\inf_{\gamma\in \Gamma}\sup_{x\in \mathbb{S}^{N-1}}I_\Omega(\gamma(x))\leq \sup_{x\in \mathbb{S}^{N-1}}I_\Omega(\bar \gamma_1(x))\leq a< c_1.
\]
This concludes the proof.
\end{proof}

\bigskip
\bigskip

\end{document}